\documentclass[10pt]{amsart}
\usepackage{amsmath,amssymb}

\newtheorem{theorem}{Theorem}[section]

\newtheorem{corollary}[theorem]{Corollary}
\newtheorem{lemma}[theorem]{Lemma}
\newtheorem{remark}[theorem]{Remark}
\newtheorem{proposition}[theorem]{Proposition}

\newtheorem{problem}[theorem]{Problem}

\numberwithin{equation}{section}

\newcommand{\R}{\mathbb{R}}
\newcommand{\N}{\mathbb{N}}
\newcommand{\Z}{\mathbb{Z}}
\newcommand{\C}{\mathbb{C}}
\newcommand{\dis}{\displaystyle}

\begin{document}

\title[The Gelfond-Schnirelman method in prime number theory]
{The Gelfond-Schnirelman method in prime number theory}%
\author{Igor E. Pritsker}%

\address{Department of Mathematics, 401 Mathematical Sciences, Oklahoma State
University, Stillwater, OK 74078-1058, U.S.A.}%
\email{igor@math.okstate.edu}

\thanks{This material is based upon work supported by the National Science
Foundation under Grant No. 9996410, and by the National Security Agency under
Grant No. MDA904-03-1-0081.}%



\begin{abstract}

The original Gelfond-Schnirelman method, proposed in 1936, uses
polynomials with integer coefficients and small norms on $[0,1]$
to give a Chebyshev-type lower bound in prime number theory. We
study a generalization of this method for polynomials in many
variables. Our main result is a lower bound for the integral of
Chebyshev's $\psi$-function, expressed in terms of the weighted
capacity. This extends previous work of Nair and Chudnovsky, and
connects the subject to the potential theory with external fields
generated by polynomial-type weights. We also solve the
corresponding potential theoretic problem, by finding the extremal
measure and its support.

\end{abstract}

\maketitle


\section{Lower bounds for arithmetic functions}

Let $\pi(x)$ be the number of primes not exceeding $x$. The
celebrated Prime Number Theorem (PNT), suggested by Legendre and
Gauss, states that
\begin{equation} \label{1,1}
\pi(x)\sim \frac{x}{\log{x}}\quad \mbox{as } x\to \infty.
\end{equation}
We include a very brief sketch of its history, referring for
details to many excellent books and surveys available on this
subject (see, e.g., \cite{Ing}, \cite{Dav}, \cite{TF} and
\cite{Dia}). Chebyshev \cite{Che} made the first important step
towards the PNT in 1852, by proving the bounds
\begin{equation} \label{1,2}
0.921\frac{x}{\log{x}}\le\pi(x)\le1.106\frac{x}{\log{x}}\quad
\mbox{as } x\to \infty.
\end{equation}
The famous Riemann's paper \cite{Rie}, published in 1859, gave a
strong impulse to the study of complex analytic methods related to
the zeta function. Thus Hadamard and de la Vall\'ee Poussin
independently proved the Prime Number Theorem in 1896, via
establishing that $\zeta(s)$ does not have zeros on the line
$\{1+it,\ t\in\R\}$. But the ``elementary" approaches to the PNT,
which do not use complex analysis and the zeta function, still
remained attractive. Selberg \cite{Sel} and Erd\H{o}s \cite{Erd}
found the first elementary proof of the Prime Number Theorem in
1949. A survey of elementary methods, with detailed history, may
be found in Diamond \cite{Dia}. The subject of this paper is the
elementary method of Gelfond and Schnirelman (see Gelfond's
comments in \cite[pp. 285--288]{Che}), proposed in 1936. Consider
the Chebyshev function
\begin{equation} \label{1.3}
\psi(x):=\sum_{p^m\le x} \log p,
\end{equation}
where the summation extends over the primes $p$. Note that
$\psi(x)=\log{\textup{lcm}(1,\ldots,x)}$ for $x\in\N.$ It is well
known that the PNT is equivalent to
\begin{equation} \label{1.4}
\psi(x) \sim x \quad \mbox{as}\ x\to +\infty
\end{equation}
(see \cite{Ing}, \cite{Dav}, \cite[Ch. 10]{Mon} and \cite{Dia}).
The idea of Gelfond and Schnirelman was based on a clever use of
polynomials with integer coefficients $p_n(x)=\dis\sum_{k=0}^n a_k
x^k$ and their integrals
\[ \int_0^1 p_n(x)\,dx=\sum_{k=0}^n\frac{a_k}{k+1}.\]
Observe that multiplying the above integral by the least common
multiple $\textup{lcm}(1,\ldots,n+1)$ gives an integer, so that
\begin{equation} \label{1.5}
\textup{lcm}(1,\ldots,n+1)\left|\int_0^1 p_n(x)\,dx\right| \ge 1,
\end{equation} provided $\int_0^1 p_n(x)\,dx \neq 0.$ Taking the
log of \eqref{1.5}, we have
\[ \psi(n+1) \ge -\log\left|\int_0^1 p_n(x)\,dx\right| \ge
-\log\max_{x\in[0,1]}|p_n(x)|.\] Hence
\begin{equation} \label{1.6}
\liminf_{n\to\infty}\frac{\psi(n+1)}{n}\ge-\log\limsup_{n\to\infty}
\left(\max_{x\in[0,1]}|p_n(x)|\right)^{1/n}.
\end{equation}
If one could find a sequence of polynomials $p_n$ with
sufficiently small sup norms $\|p_n\|_{[0,1]},$ so that
\begin{equation} \label{1.7}
\lim_{n \rightarrow \infty} \| p_n \|_{[0,1]}^{1/n}
\stackrel{?}{=} 1/e,
\end{equation}
then the PNT followed from \eqref{1.6}. A nice account on the
original Gelfond-Schnirelman attempt is contained in Montgomery
\cite[Ch. 10]{Mon} (also see Chudnovsky \cite{Chu}). We are led by
this method to the so-called integer Chebyshev problem on
polynomials with integer coefficients minimizing the sup norm
(see, e.g., Borwein \cite{Bor}). Let $\Z_n[x]$ be the set of
polynomials over integers, of degree at most $n$. In view of
\eqref{1.6}-\eqref{1.7}, we are interested in the integer
Chebyshev constant
\begin{equation} \label{1.8}
t_{\Z}([0,1]) := \lim_{n \rightarrow \infty} \left(\inf_{0
\not\equiv p_n \in\Z_n[x]} \| p_n \|_{[0,1]}\right)^{1/n}.
\end{equation}
It was found by Gorshkov \cite{Gor} in 1956 that \eqref{1.7} can
never be achieved. In fact, $0.4213 < t_{\Z}([0,1]) < 0.4232$ (see
\cite{Pri} for a survey of recent results on this problem). Thus
the Gelfond-Schnirelman method failed in its original form, but
one can generalize it for polynomials in many variables. Such an
idea apparently had first appeared in Trigub \cite{Tri}, and was
independently implemented by Nair \cite{Nai} and Chudnovsky
\cite{Chu}. The basis of their argument lies in another equivalent
form of the Prime Number Theorem \cite{Ing}:
\begin{equation} \label{1.9}
\int_1^x \psi(t)\ dt \sim \frac{x^2}{2} \quad \mbox{as}\ x\to
+\infty.
\end{equation}
Both Nair and Chudnovsky used the following weighted version of
Vandermonde determinant
\begin{align} \label{1.10}
V^w_n(x_1,\ldots,x_n) &:= \prod_{1\le i<j \le n} (x_i-x_j)
w(x_i)w(x_j) \\ \nonumber &= \prod_{i=1}^n w^{n-1}(x_i)
\prod_{1\le i<j \le n} (x_i-x_j),
\end{align}
where $x_i\in [0,1]$ and $w(x)=(x(1-x))^{\alpha_1},\ \alpha_1>0$,
to generate multivariate polynomials with small sup norms on the
cube $[0,1]^n.$ They obtained the numerical bound
\begin{equation} \label{1.11}
\int_1^x \psi(t)\ dt \ge 0.99035\, \frac{x^2}{2} \quad \mbox{as}\
x\to +\infty,
\end{equation}
produced by the optimal choice $\alpha_1\approx 0.195$ (our
notations differ from those of \cite{Nai} and \cite{Chu}).
Chudnovsky \cite{Chu} also indicated how this approach can be
generalized for the weights of the form
\begin{equation} \label{1.12}
w(x)=\prod_{i=1}^k |Q_{m_i}(x)|^{\alpha_i},
\end{equation}
where $Q_{m_i} \in \Z_{m_i}[x]$ and $\alpha_i>0,\ i=1,\ldots,k$.
We develop the ideas of \cite{Nai} and \cite{Chu}, and establish a
connection with the weighted potential theory (or potential theory
with external fields) that originated in the work of Gauss
\cite{Gau} and Frostman \cite{Fro} (see \cite{ST} for a modern
account on this theory). An important part of the method is the
analysis of the asymptotic behavior for the supremum norms of the
weighted Vandermonde determinants \eqref{1.10}, which is governed
by the weighted capacity $c_w$ of $[0,1]$ corresponding to the
weight $w$ (cf. Section 2 below and \cite{ST}). This method leads
to the following lower bound for the integral of $\psi$-function
via $c_w.$

\begin{theorem} \label{thm1.1}
Let $w(x)$ be as in \eqref{1.12} and let $\alpha:=\sum_{i=1}^k
\alpha_im_i.$ Then
\begin{equation} \label{1.13}
\int_1^x \psi(t)\ dt \ge \frac{-2 \log{c_w}}{4\alpha+3} \,
\frac{x^2}{2} + O(x\log^2{x}) \quad \mbox{as}\ x\to +\infty.
\end{equation}
\end{theorem}

We recover the results of Nair and Chudnovsky as a special case of
Theorem \ref{thm1.1}.

\begin{corollary} \label{cor1.2}
If $w(x)=x^{\alpha_1}(1-x)^{\alpha_2},\ x\in [0,1],\
\alpha_1=\alpha_2=0.195,$ then $c_w\approx 0.1045575588$ and
\eqref{1.11} holds true.
\end{corollary}

It is natural to try improving the bound \eqref{1.11} by choosing
a weight with a proper combination of factors $Q_{m_i}(x)$ and
exponents $\alpha_i$. The most interesting question is, of course,
whether one can find a weight $w(x)$ of the form \eqref{1.12} such
that
\[
\frac{-2 \log{c_w}}{4\alpha+3} = 1?
\]
It turns out this is impossible to achieve for any {\em fixed}
weight of the type \eqref{1.12}. The reason for such a conclusion
transpires from the error term in \eqref{1.13}, which is ``too
good." Indeed, it is known from Littlewood's theorem that the
difference $\int_1^x \psi(t)\ dt - x^2/2$ takes both positive and
negative values of the amplitude $c x^{3/2},\ c>0,$ infinitely
often as $x\to +\infty.$ This is conveniently written in the
notation
\[
\int_1^x \psi(t)\ dt - \frac{x^2}{2} = \Omega_{\pm}(x^{3/2}) \quad
\mbox{as } x \to +\infty
\]
(cf. \cite[pp. 91-92]{Ing}). Hence the correct error term should
be of the order $O(x^{3/2}).$ Relating this to \eqref{1.9} and
\eqref{1.13}, we obtain in such an indirect way the following.

\begin{proposition} \label{prop1.2}
Given a weight $w(x)$ of the form \eqref{1.12}, we have
\begin{equation} \label{1.14}
B(w):=\frac{-2 \log{c_w}}{4\alpha+3} < 1,
\end{equation}
where $\alpha=\sum_{i=1}^k \alpha_im_i.$
\end{proposition}

We should also note that if the Riemann hypothesis is true, then
\[
\int_1^x \psi(t)\ dt - \frac{x^2}{2} = O(x^{3/2}) \quad \mbox{as }
x \to +\infty
\]
(see Theorem 30 in \cite[p. 83]{Ing}). It would be very
interesting to find a direct potential theoretic argument
explaining \eqref{1.14}. Although \eqref{1.13} cannot provide a
proof of the PNT for a fixed weight $w$, this does not preclude
the possibility that such a proof can be obtained by finding a
{\em sequence} of weights $w_n$ with $B(w_n)\to 1$, as $n\to
\infty.$ On the other hand, we did not observe a numerical
improvement of the estimate \eqref{1.11} when using further
factors of the one-dimensional integer Chebyshev polynomials for
the weight $w$, beyond the factors $x$ and $1-x$ (see \cite{Mon},
\cite{Chu} and \cite{Pri}). Thus one needs a better insight into
the arithmetic nature of such factors, to address the problem
stated below.

\begin{problem} \label{prob1.3}
For $w(x)$ as in \eqref{1.12} and $\alpha=\sum_{i=1}^k
\alpha_im_i,$ find
\begin{equation} \label{1.15}
B:=\sup_w \frac{-2 \log{c_w}}{4\alpha+3}.
\end{equation}
If $B=1$ then find a sequence of weights that gives this value. If
$B<1$ then investigate whether $B$ is attained for a weight of the
form \eqref{1.12}.
\end{problem}

The solution of this problem also requires a detailed knowledge of
the potential theory with external fields generated by the weights
\eqref{1.12}, which is discussed in the following section.

We remark that the scope of the multivariate Gelfond-Schnirelman
method still remains much wider than the approach proposed by Nair
and Chudnovsky. Indeed, the weighted Vandermonde determinant
$V^w_n(x_1,\ldots,x_n)$ of \eqref{1.10} is a very special case of
a multivariate polynomial with small norm, which is unlikely best
possible. It is of great interest to design other sequences of
polynomials providing good bounds for the arithmetic functions,
along the discussed lines. Among the natural candidates are the
multivariate Vandermonde determinants (see \cite{Zah} and
\cite{BC99}) and other sequences of minimal polynomials
\cite{BC00}. This subject is closely related to pluripotential
theory \cite{Kli}.

\section{Potential theory with external fields}

We consider a special case of the weighted energy problem on a
segment of the real line $[a,b]$, which is associated with the
``polynomial-type" weights \eqref{1.12}. A comprehensive treatment
of the potential theory with external fields, or weighted
potential theory, is contained in the book of Saff and Totik
\cite{ST}, together with historical remarks and numerous
references. It is convenient to rewrite the weight function in the
following more general form:
\begin{align} \label{2.1}
w(x)=A\prod_{i=1}^K |x-z_i|^{p_i}, \qquad x \in [a,b],
\end{align}
where $A>0,\ p_i>0$ and $z_i\in\C$. Let ${\mathcal M}([a,b])$ be
the set of positive unit Borel measures supported on $[a,b]$. For
any measure $\mu\in {\mathcal M}([a,b])$ and weight $w$ of
\eqref{2.1}, we define the energy functional
\begin{align} \label{2.2}
I_{w} (\mu)&:= \int\!\!\int \log \dis\frac{1}{|z-t|w(z)w(t)} \ d
\mu(z)d \mu(t) \\ \nonumber &= \int\!\!\int \log \dis
\frac{1}{|z-t|} \ d \mu(z)d \mu(t) - 2 \int \log w(t)\, d\mu(t),
\end{align}
and consider the minimum energy problem
\begin{equation} \label{2.3}
V_{w}:= \dis\inf_{\mu \in {\mathcal M}([a,b])} I_{w}(\mu).
\end{equation}
It follows from Theorem I.1.3 of \cite{ST} that  $V_{w}$ is
finite, and there exists a unique equilibrium measure $\mu_{w} \in
{\mathcal{M}}([a,b])$ such that $I_{w} (\mu_{w}) = V_{w}$. Thus
$\mu_w$ minimizes the energy functional \eqref{2.2} in presence of
the external field generated by the weight $w$. Furthermore, we
have for the potential of $\mu_w$ that
\begin{equation} \label{2.4}
U^{\mu_{w}}(x)-\log w(x) \geq F_{w}, \qquad x \in [a,b],
\end{equation}
and
\begin{equation} \label{2.5}
U^{\mu_{w}}(x)-\log w(x) = F_{w}, \qquad x \in S_w,
\end{equation}
where $U^{\mu_w}(x):=-\dis\int\log|x-t|\,d\mu_w(t), \ F_{w}:=V_{w}
+ \dis\int \log w(t) d\mu_{w}(t)$ and $S_w:={\rm supp}\, \mu_{w}$
(see Theorems I.1.3 and I.5.1 in \cite{ST}) . The weighted
capacity of $[a,b]$ is defined by
\begin{equation} \label{2.6}
\textup{cap}([a,b],w):=e^{-V_w}.
\end{equation}
In agreement with the notation of Section 1, we set
\[
c_w:=\textup{cap}([0,1],w).
\]
If $w\equiv 1$ on $[a,b]$, then we obtain the classical
logarithmic capacity $\textup{cap}([a,b],1)= (b-a)/4$ (cf.
\cite{Ran}).

The support $S_w$ plays a crucial role in determining the
equilibrium measure $\mu_w$ itself, as well as other components of
this weighted energy problem. Indeed, if $S_w$ is known then
$\mu_w$ can be found as a solution of the singular integral
equation
\[
\int\log\frac{1}{|x-t|}\,d\mu(t)-\log w(x) = F, \qquad x \in S_w,
\]
where $F$ is a constant (cf. \eqref{2.5} and \cite[Ch. IV]{ST}).
For $w$ given by \eqref{2.1} or \eqref{1.12}, this equation can be
solved by potential theoretic methods, using balayage techniques,
so that $\mu_w$ is expressed as a linear combination of harmonic
measures (see Lemma \ref{lem3.3} and \cite{Pri}). We follow
another path here, via the methods of singular integral equations,
which gives a more explicit solution. This approach for
polynomial-type weights was suggested by Chudnovsky \cite{Chu} and
developed further by Amoroso \cite{Amo}. For more general weights,
one should consult Chapter IV of \cite{ST} and the paper of Deift,
Kreicherbauer and McLaughlin \cite{DKM}. We give an explicit form
of the equilibrium measure and describe its support in the
following result.

\begin{theorem} \label{thm2.1}
Let $Z:=\bigcup_{i=1}^K \{z_i\}\subset [a,b]$ be the set of zeros
for $w$ of \eqref{2.1}, where $z_1=a<z_2<\ldots<z_K=b$. There
exist an integer $L,\ 1 \le L \le K-1,$ a polynomial $P(x) =
x^{K-L-1}+\ldots\in\R_{K-L-1}[x]$, and $L$ intervals
$[a_l,b_l]\subset [a,b]\setminus Z$, with
$a<a_1<b_1<a_2<b_2<\ldots<a_L<b_L<b$,  such that
\begin{equation} \label{2.7}
S_w=\bigcup_{l=1}^L\, [a_l,b_l],
\end{equation}
and the equilibrium measure $\mu_w$ is given by
\begin{equation} \label{2.8}
d\mu_{w}(x)=(-1)^{L+l+1}\ \frac{(1+p) \sqrt{|R(x)|}\,P(x)}{\pi
\prod_{j=1}^K (x-z_j)}\, dx, \qquad x \in [a_l,b_l],
\end{equation}
where $l=1,\ldots,L,$ $p:=\sum_{j=1}^K p_j$ and
$R(x):=\prod_{l=1}^L (x-a_l)(x-b_l)$.

Furthermore, the polynomial $P(x)$ and the endpoints of $S_w$
satisfy the equations
\begin{equation} \label{2.9}
P(z_j)=(-1)^{L+l}\ \frac{p_j\prod_{m\neq j}
(z_j-z_m)}{(1+p)\sqrt{|R(z_j)|}}, \qquad z_j\in [b_l,a_{l+1}],\
j=1,\ldots,K,
\end{equation}
where we set $b_0=-\infty,\ a_{L+1}=+\infty,$ and the equations
\begin{equation} \label{2.10}
\int_{b_l}^{a_{l+1}} \frac{\sqrt{|R(x)|}\,P(x)}{ \prod_{j=1}^K
(x-z_j)}\, dx=0, \qquad l=1,\ldots,L-1,
\end{equation}
where the integrals are understood as Cauchy principal values.
\end{theorem}

Recall that we need the quantity $-\log \textup{cap}([a,b],w)=V_w$
for Theorem \ref{thm1.1}. This can be found from \eqref{2.5} as
\[
V_w=U^{\mu_{w}}(x)-\log w(x)-\int\log w \,d\mu_w,
\]
for any $x\in S_w.$

The assumption that the weight $w$ vanishes at the endpoints of
$[a,b]$ seems appropriate in this case, due to the role of factors
$x$ and $1-x$, for $w(x)$ on $[0,1]$, in the work of Nair and
Chudnovsky. Other cases of weights \eqref{2.1} with real zeros can
be handled similarly, along the lines of this paper. Perhaps, a
more interesting problem is to consider weights with complex
zeros, when $-\log w(x)$ is not piecewise convex on $[a,b]$.

If the support $S_w$ consists of $L=K-1$ intervals, then
$P(x)\equiv 1$ in Theorem \ref{thm2.1}, and we obtain the
following result.

\begin{corollary} \label{cor2.2}
If $L=K-1$ in Theorem \ref{thm2.1}, then
\begin{equation} \label{2.11}
d\mu_{w}(x)=\frac{(1+p) \sqrt{|R(x)|}}{\pi \prod_{j=1}^K
|x-z_j|}\, dx, \qquad x \in S_w.
\end{equation}
Furthermore, the following equations hold true:
\begin{equation} \label{2.12}
\sqrt{|R(z_j)|} = \frac{p_j}{1+p} \prod_{m\neq j} |z_j-z_m|,
\qquad j=1,\ldots,K,
\end{equation}
and
\begin{equation} \label{2.13}
\int_{b_l}^{a_{l+1}} \frac{\sqrt{|R(x)|}}{ \prod_{j=1}^K
(x-z_j)}\, dx=0, \qquad l=1,\ldots,K-2.
\end{equation}
\end{corollary}

\begin{remark} \label{rem2.3}
\textup{Amoroso \cite{Amo} studied the weighted energy problem for
the weights of Theorem \ref{thm2.1} (expressed in slightly
different terms). In particular, Theorem 2.2 of \cite{Amo} states
that the equilibrium measure always has the form \eqref{2.11},
which is not true as shown in the next section. The main
shortcoming is in the assumption of \cite{Amo} that the support
$S_w$ {\em always} consists of $K-1$ intervals, one less than the
number of zeros of $w$. Also, \cite{Amo} defines the endpoints of
the support as solutions of \eqref{2.12}-\eqref{2.13}, without
analyzing that these solutions exist and fall within the needed
range. Assuming that \eqref{2.12}-\eqref{2.13} hold true, the
representation \eqref{2.11} is then deduced from those equations
in \cite{Amo}. We emphasize that the results of \cite{Amo} on the
irrationality measures of logarithms, stated in Theorems 4.2 and
4.3, are correct. In those applications, the parameters are
selected so that \eqref{2.12}-\eqref{2.13} are valid, and the
support indeed has $K-1$ intervals.}
\end{remark}

Note that equations \eqref{2.9}-\eqref{2.10} (correspondingly,
\eqref{2.12}-\eqref{2.13}) may be used to find the coefficients of
$P(x)$ and the endpoints of $S_w$. Thus we have $K-L-1$
coefficients and $2L$ endpoints to find, which gives $K+L-1$
unknowns and the same number of equations
\eqref{2.9}-\eqref{2.10}. For example, if $K=2$, i.e., we have the
so-called Jacobi-type weight $w$ on $[a,b]$, then $L=1$ and
\eqref{2.12} gives just two equations for the endpoints of
$S_w=[a_1,b_1]$. It is easy to solve them explicitly, and find the
well known representation for $S_w$ and $\mu_w$ from Corollary
\ref{cor2.2} (see, e.g., Examples IV.1.17 and IV.5.2 of
\cite{ST}).

\begin{corollary} \label{cor2.4}
Suppose $w(x)=x^{p_1} (1-x)^{p_2}, \ x\in[0,1],$ where
$p_1,p_2>0.$ Then
\begin{equation} \label{2.14}
d\mu_{w}(x)=\frac{(1+p_1+p_2) \sqrt{(x-a_1)(b_1-x)}}{\pi x(1-x)}\,
dx, \qquad x \in [a_1,b_1].
\end{equation}
The endpoints of the support are
\begin{equation} \label{2.15}
a_1=\frac{1+r_1^2-r_2^2-\sqrt{(1+r_1^2-r_2^2)^2-4r_1^2}}{2}
\end{equation}
and
\begin{equation} \label{2.16}
b_1=\frac{1+r_1^2-r_2^2+\sqrt{(1+r_1^2-r_2^2)^2-4r_1^2}}{2},
\end{equation}
where we set
\[
r_1:=\frac{p_1}{1+p_1+p_2} \quad \mbox{and} \quad
r_2:=\frac{p_2}{1+p_1+p_2}.
\]

\end{corollary}

\smallskip

The connection between the potential theory with external fields
and this version of the Gelfond-Schnirelman method arises in the
need for asymptotics of the weighted Vandermonde determinant
\eqref{1.10}. It is known that
\begin{equation} \label{2.30}
\lim_{n\to \infty}\left(\max_{x_1,\dots,x_n\in [a,b]}
|V_n^w(x_1,\ldots,x_n)|\right)^{\frac{2}{n(n-1)}} =
\textup{cap}([a,b],w)
\end{equation}
(see Theorem III.1.3 of \cite{ST}). The quantity on the left-hand
side of \eqref{2.30} is called the weighted transfinite diameter
of $[a,b]$. In the case $w\equiv 1$, it was introduced by Fekete
\cite{Fek} for arbitrary compact sets in the plane. Szeg\H{o}
\cite{Sze} showed that the transfinite diameter coincides with the
logarithmic capacity, so that \eqref{2.30} is a generalization of
his result. We quantify the rate of convergence in \eqref{2.30}.
\begin{lemma} \label{lem2.1}
Let $w$ be as in \eqref{2.1}. There exist constants $d=d(w)>1$ and
$D=D(w)>0$ such that
\begin{align} \label{2.31}
\left(\textup{cap}([a,b],w)\right)^{n(n-1)} \le \max_{[a,b]^n}
(V^w_n)^2 \le D\,d^{n\log^2{n}}\,
\left(\textup{cap}([a,b],w)\right)^{n(n-1)}.
\end{align}
\end{lemma}
Equation \eqref{2.31} is the only fact from potential theory
needed in the proof of Theorem \ref{thm1.1}. It is very likely
that $\log^2{n}$ can be replaced by $\log{n}$, matching the
classical case (see Theorem 1.3.3 in \cite{AB}). This would give a
corresponding improvement in the error term of \eqref{1.13}, but
we do not pursue this direction.

\section{Proofs}

\begin{proof}[Proof of Theorem \ref{thm1.1}]

The proof is based on an argument similar to the original
Gelfond-Schnirelman idea (cf. \cite{Nai} and \cite{Chu}). We
consider the integrals of small polynomials with integer
coefficients over the cube $[0,1]^n,\ n\in\N.$ It is important
that the integrals are non-zero, so that we work with the square
of the weighted Vandermonde determinant \eqref{1.10}, instead:
\begin{align} \label{5.1}
\Delta^w_n(x_1,\ldots,x_n) &:= \left( V^w_n \right)^2 =
\prod_{i=1}^n w^{2(n-1)}(x_i) \prod_{1\le i<j \le n} (x_i-x_j)^2.
\end{align}
If $w(x)\equiv 1$ then $\Delta^w_n$ is the classical discriminant.
In general, $\Delta^w_n$ is not a polynomial in $x_i$'s because of
the real exponents $\alpha_i$'s in the weight \eqref{1.12}. Hence
we modify it further into
\begin{align} \label{5.2}
\tilde\Delta^w_n(x_1,\ldots,x_n) &:= \prod_{j=1}^n \prod_{i=1}^k
\left(Q_{m_i}(x_j)\right)^{2\lceil\alpha_i(n-1)\rceil} \prod_{1\le
i<j \le n} (x_i-x_j)^2,
\end{align}
where $\lceil a\rceil$ denotes the ceiling function: the smallest
integer at least $a$. It is now clear that
$\tilde\Delta^w_n(x_1,\ldots,x_n)$ is a positive polynomial with
integer coefficients that has the following form
\begin{align} \label{5.3}
\tilde\Delta^w_n(x_1,\ldots,x_n) &= \sum a_{i_1\ldots i_n}
x_1^{i_1}\ldots x_n^{i_n}.
\end{align}
Recall the definition of the classical Vandermonde determinant
\begin{align*}
V_n := \left|
\begin{array}{cccc}
  1 & x_1 & \ldots & x_1^{n-1} \\
  1 & x_2 & \ldots & x_2^{n-1} \\
  \vdots & \vdots & \vdots & \vdots \\
  1 & x_n & \ldots & x_n^{n-1}
\end{array}
\right|  = \prod_{1\le i<j \le n} (x_i-x_j).
\end{align*}
Using standard expansion of $V_n$, we observe that each term of
this expansion is a (signed) product of all powers of $x_i$'s from
$0$ to $n-1$. The expression in \eqref{5.2} is equal to $V_n^2$
times the weight part. Thus if we arrange the powers
$i_1,\ldots,i_n$ in every term of \eqref{5.3} in the increasing
order, we have that
\begin{align} \label{5.4}
i_j\le n+j-2+N,\quad j=1,\ldots,n,
\end{align}
where $N:=2\sum_{i=1}^k m_i \lceil\alpha_i(n-1)\rceil$ is the
contribution of the weight part in \eqref{5.2}. Hence
\[
\int_0^1\ldots\int_0^1 \tilde\Delta^w_n(x_1,\ldots,x_n)\,
dx_1\ldots dx_n = \sum \frac{a_{i_1\ldots i_n}}{(i_1+1)\ldots
(i_n+1)} \neq 0
\]
is a rational number whose denominator divides
$\prod_{l=n+N}^{2n-1+N} \textup{lcm}(1,\ldots,l)$ by \eqref{5.4}.
It follows as in \eqref{1.5} that
\[
\prod_{l=n+N}^{2n-1+N} \textup{lcm}(1,\ldots,l) \int_{[0,1]^n}
\tilde\Delta^w_n \ge 1.
\]
On taking the logarithm, we obtain that
\[
\sum_{l=n+N}^{2n-1+N} \psi(l) \ge -\log \int_{[0,1]^n}
\tilde\Delta^w_n \ge -\log \max_{[0,1]^n} \tilde\Delta^w_n.
\]
It is clear from \eqref{5.1} and \eqref{5.2} that
\[
\tilde\Delta^w_n = \Delta^w_n \prod_{j=1}^n \prod_{i=1}^k
|Q_{m_i}(x_j)|^{2(\lceil\alpha_i(n-1)\rceil-\alpha_i(n-1))},
\]
which gives
\[
\max_{[0,1]^n} \tilde\Delta^w_n \le \max_{[0,1]^n} \Delta^w_n \
\prod_{i=1}^k \left(\max(1,\|Q_{m_i}\|_{[0,1]})\right)^{2n}.
\]
Since $\psi(x)$ is constant between integers, we arrive at the
estimate
\begin{align} \label{5.5}
\int_{n+N}^{2n+N} \psi(y)\, dy \ge -\log \max_{[0,1]^n} \Delta^w_n
+ O(n) \quad \mbox{as } n \to \infty.
\end{align}
We now need the following consequence of Lemma \ref{lem2.1}:
\[
\log \max_{[0,1]^n} \Delta^w_n = n^2 \log{c_w} + O(n\log^2{n})
\quad \mbox{as } n \to \infty.
\]
Applying this in \eqref{5.5}, we have
\[
\int_{n+N}^{2n+N} \psi(y)\, dy \ge - n^2 \log c_w + O(n\log^2 n)
\quad \mbox{as } n \to \infty.
\]
Note that $2\alpha (n-1) \le N \le 2\alpha (n-1) + 2\sum_{i=1}^k
m_i.$ If we set $2n(\alpha+1)=x$, then
\begin{align} \label{5.6}
\int_{\frac{2\alpha+1}{2\alpha+2}x}^x \psi(y)\, dy \ge -
\frac{\log c_w}{4(\alpha+1)^2}\, x^2 + O(x\log^2 x) \quad \mbox{as
} x \to \infty.
\end{align}
Using the substitution $x\to \frac{2\alpha+1}{2\alpha+2}\,x$
iteratively and summing up the results, we obtain
\begin{align*}
\int_{1}^x \psi(y)\, dy \ge  \frac{-\log{c_w}}{4\alpha+3} \, x^2 +
O(x\log^2{x}) \quad \mbox{as } x \to \infty.
\end{align*}

\end{proof}

\begin{proof}[Proof of Corollary \ref{cor1.2}]

Corollary \ref{cor2.4} gives the weighted equilibrium measure
$\mu_w$ for such weights in \eqref{2.14}-\eqref{2.16}. Hence we
have by \eqref{2.5} that the numerical value of $c_w$ can be
computed from
\[
-\log c_w = U^{\mu_{w}}(a_1)-\log w(a_1)-\int\log w \,d\mu_w,
\]
where $a_1$ is defined in \eqref{2.15}. The same equation yields
\eqref{1.11}, as a consequence of Theorem \ref{thm1.1}.

\end{proof}

\smallskip

We now start preparations for the proof of Theorem \ref{thm2.1}.
Recall the function $R(z)=\prod_{l=1}^L (z-a_l)(z-b_l)$, where
$a_1<b_1<a_2<b_2<\ldots<a_L<b_L$ are real numbers. The branch of
$\sqrt{R(z)}$, satisfying $\lim_{z\to\infty} \sqrt{R(z)}/z^L = 1$,
is analytic in $\C\setminus \bigcup_{l=1}^L [a_l,b_l]$. For
further reference, we describe the values of $\sqrt{R(z)}$ on the
real line:
\begin{equation} \label{5.7}
\sqrt{R(x)} = \left\{
\begin{array}{ll}
\sqrt{|R(x)|}, \quad &x\ge b_L, \\
(-1)^{L+l}\, i\, \sqrt{|R(x)|}, \quad &a_l\le x \le b_l,\ l=1,\ldots,L, \\
(-1)^{L+l}\, \sqrt{|R(x)|}, \quad &b_l\le x \le a_{l+1},\ l=1,\ldots,L-1, \\
(-1)^L\, \sqrt{|R(x)|}, \quad &x \le a_1.
\end{array}
\right.
\end{equation}
Here and throughout, the values of $\sqrt{R(x)}$ for $x \in
\bigcup_{l=1}^L [a_l,b_l]$ are understood as the upper limiting
values of $\sqrt{R(z)}$, when $\Im{z}\to 0^+.$

\begin{lemma} \label{lem3.1}
Let $S:=\bigcup_{l=1}^L\, [a_l,b_l]$. For any
$T_{L-1}\in\R_{L-1}[x]$, we have
\begin{equation} \label{5.8}
\frac{1}{\pi i} \int_S \frac{T_{L-1}(t)\,dt}{(t-z) \sqrt{R(t)}} =
\left\{
\begin{array}{ll}
0, \quad &z\in S, \\
T_{L-1}(z)/\sqrt{R(z)}, \quad &z \in \C\setminus S.
\end{array}
\right.
\end{equation}
\end{lemma}

\begin{proof}

A detailed proof of this known fact may be found in \cite{Mus}
(cf. Chapter 11). We give a sketch of argument based on Cauchy
integral formula. Consider a contour $\Gamma$ that consists of $L$
simple closed curves around each of the intervals $[a_l,b_l]$,
located close to those intervals. Then
\begin{equation*}
\frac{1}{2\pi i} \oint_{\Gamma} \frac{T_{L-1}(t)\,dt}{(t-z)
\sqrt{R(t)}} = \left\{
\begin{array}{ll}
0, \quad &z\in S, \\
T_{L-1}(z)/\sqrt{R(z)}, \quad &z \in \textup{Ext}(\Gamma).
\end{array}
\right.
\end{equation*}
Note that the limiting values of $\sqrt{R(t)}$ on $S$, from above
and below, are opposite in sign. Letting the contour $\Gamma$
shrink to $S$, we obtain that
\[
\lim_{\Gamma\to S} \frac{1}{2\pi i} \oint_{\Gamma}
\frac{T_{L-1}(t)\,dt}{(t-z) \sqrt{R(t)}} = \frac{1}{\pi i} \int_S
\frac{T_{L-1}(t)\,dt}{(t-z) \sqrt{R(t)}}.
\]

\end{proof}

Let $\omega(z,\cdot,\Omega)$ be the harmonic measure at $z \in
\Omega$ with respect to a domain $\Omega\subset\overline\C$, which
is a positive unit Borel measure supported on $\partial\Omega$.
The background information on harmonic measures may be found in
Section 4.3 of \cite{Ran} and Appendix A.3 of \cite{ST}. In
particular, $\omega(z,\cdot,\Omega)$ is equal to the balayage of
the unit point mass $\delta_z$ to $\partial\Omega.$ We give the
following explicit representations for these measures.

\begin{lemma} \label{lem3.2}
Let $\Omega:=\overline{\C}\setminus\bigcup_{l=1}^L\, [a_l,b_l]$.
There exist polynomials $T_{\infty}\in\R_{L-1}[x]$ and
$T_{z_j}\in\R_{L-1}[x],\ j=1,\ldots,K$, such that
\begin{equation} \label{5.9}
d\,\omega(\infty,x,\Omega)=\frac{T_{\infty}(x)\,dx}{\pi i
\sqrt{R(x)}}, \qquad x \in \bigcup_{l=1}^L\, [a_l,b_l],
\end{equation}
and
\begin{equation} \label{5.10}
d\,\omega(z_j,x,\Omega)=\frac{T_{z_j}(x)\,dx}{\pi i (x-z_j)
\sqrt{R(x)}}, \qquad x \in \bigcup_{l=1}^L\, [a_l,b_l],
\end{equation}
where $z_j\in\Omega\setminus\{\infty\},\ j=1,\ldots,K.$
\end{lemma}

\begin{proof}

These formulas are essentially known (see, e.g., Lemma 4.4.1 of
\cite{StT} and Lemma 2.3 of \cite{Tot}). It is possible to deduce
\eqref{5.9}-\eqref{5.10} from Lemma \ref{lem3.1}, which is done
below.

We select $T_{\infty}(t)=\sum_{j=1}^{L-1} c_j t^j \in \R_{L-1}[t]$
so that it satisfies the following equations:
\begin{equation} \label{5.11}
\int_{b_l}^{a_{l+1}} \frac{T_{\infty}(t)\,dt}{\sqrt{R(t)}} =
\sum_{j=1}^{L-1} c_j \int_{b_l}^{a_{l+1}}
\frac{t^j\,dt}{\sqrt{R(t)}} = 0, \quad l=1,\ldots,L-1,
\end{equation}
and
\begin{equation} \label{5.12}
\frac{1}{\pi i} \int_S \frac{T_{\infty}(t)\,dt}{\sqrt{R(t)}} =
\sum_{j=1}^{L-1} \frac{c_j}{\pi i} \int_S
\frac{t^j\,dt}{\sqrt{R(t)}} = 1,
\end{equation}
where $S=\bigcup_{l=1}^L\, [a_l,b_l].$ The polynomial
$T_{\infty}(t)$ is defined by these equations uniquely, because
the corresponding homogeneous system of linear equations (with
zero on the right of \eqref{5.12}), in the coefficients $c_j$ of
$T_{\infty}(t)$, has only trivial solution. Indeed, let $T_h(t)$
be a nontrivial solution of this homogeneous system. Since the
sign of $\sqrt{R(t)}$ is constant on each $(b_l,a_{l+1})$, by
\eqref{5.7}, $T_h(t)$ must change sign on each $[b_l,a_{l+1}],\
l=1,\ldots,L-1,$ by \eqref{5.11}. Hence $T_h(t)$ has a simple zero
in each $(b_l,a_{l+1}),\ l=1,\ldots,L-1,$ and it alternates sign
on the intervals $[a_l,b_l],\ l=1,\ldots,L.$ (Note that the same
is true for $T_{\infty}(t).$) It follows from \eqref{5.7} that
$T_h(t)/(\pi i \sqrt{R(t)})$ doesn't change sign on $S,$
contradicting
\[
\frac{1}{\pi i} \int_S \frac{T_h(t)\,dt}{\sqrt{R(t)}} = 0.
\]
Thus $T_{\infty}(t)$ exists and is unique. In addition, the above
argument and \eqref{5.12} show that $T_{\infty}(t)/(\pi i
\sqrt{R(t)})$ keeps positive sign on $S,$ i.e., \eqref{5.9}
actually defines a positive unit Borel measure on $S.$ Let
\[
h(x):=\frac{1}{\pi i} \int_S \frac{T_{\infty}(t)\,dt}{(t-x)
\sqrt{R(t)}}, \qquad x\in\R.
\]
We have from \eqref{5.8} that $h\in L_p([a_1,b_L]),\ 1\le p<2,$
and $h\in C^{\infty}\left( [a_1,b_L]\setminus\{a_l,b_l\}_{l=1}^L
\right).$ Note that $h(x)$ is the derivative of potential for the
measure on the right of \eqref{5.9}. The fundamental theorem of
calculus gives that
\[
\frac{1}{\pi i} \int_S \log \frac{1}{|x-t|}
\frac{T_{\infty}(t)\,dt}{\sqrt{R(t)}} = \int_{a_1}^x h(t)\,dt + C,
\qquad x\in[a_1,b_L],
\]
where $C$ is a constant. Using \eqref{5.8} and \eqref{5.11}, we
obtain that
\[
\frac{1}{\pi i} \int_S \log \frac{1}{|x-t|}
\frac{T_{\infty}(t)\,dt}{\sqrt{R(t)}} = C, \qquad x\in S.
\]
Frostman's theorem (see Theorem 3.3.4 of \cite{Ran}) and the
uniqueness of the equilibrium measure (cf. Theorem 3.7.6 of
\cite{Ran}) imply that $T_{\infty}(x)\,dx/(\pi i \sqrt{R(x)})$ is
the classical (not weighted) equilibrium measure for $S$. The
latter is known to be equal to $\omega(\infty,x,\Omega)$ (see
Theorem 4.3.14 of \cite{Ran}), which proves \eqref{5.9}. Equations
\eqref{5.10} follow from \eqref{5.9} by using the transformations
\[
f_j(z) := \frac{1}{z-z_j}, \qquad z\in\overline\C,\ j=1,\ldots,K,
\]
and the conformal invariance of harmonic measures (cf. Theorem
4.3.8 of \cite{Ran}). Indeed, if $\Omega_j:=f_j(\Omega)$ and
$S_j:=f_j(S)=\partial\Omega_j,$ then
\[
\omega(z_j,x,\Omega)=\omega(\infty,(x-z_j)^{-1},\Omega_j), \qquad
 j=1,\ldots,K.
\]
We can use \eqref{5.9} on the right side of the above equation,
because $S_j$ is also a union of $L$ segments of the real line,
which gives \eqref{5.10}.

\end{proof}

In \cite{Pri}, we showed that the weighted energy problem
discussed in Section 2 can be solved in terms of a linear
combination of harmonic measures.

\begin{lemma} \label{lem3.3}
Let $w(x)$ be as in \eqref{2.1}, $x\in [a,b]$. Then $S_w \subset
[a,b] \setminus \{z_j\}_{j=1}^K,$ and the extremal measure for the
weighted energy problem \eqref{2.2}-\eqref{2.3} is given by
\begin{equation} \label{5.13}
\mu_w = (1+p)\omega(\infty,\cdot,\Omega) - \sum_{j=1}^K p_j
\omega(z_j,\cdot,\Omega),
\end{equation}
where $\Omega:=\overline{\C} \setminus S_w$ and $p=\sum_{j=1}^K
p_j.$
\end{lemma}

\begin{proof}
The existence of a weighted equilibrium measure $\mu_w$, whose
support is a compact set $S_w \subset [a,b] \setminus
\{z_j\}_{j=1}^K$, follows from Theorem I.1.3 of \cite{ST}. Let
$\delta_z$ be a unit point mass at $z$. Observe that
\begin{equation} \label{5.14}
\log w(z)=\log A - U^{\nu}(z), \qquad z\in\C,
\end{equation}
where $U^{\nu}$ is the logarithmic potential of the measure
\[
\nu:= \sum_{j=1}^K p_j \delta_{z_j}.
\]
It is clear that $\nu$ is a positive Borel measure of total mass
$\nu(\C)=p.$ Let $\hat\nu$ be the balayage of $\nu$ from $\Omega$
onto $S_w$ (see, e.g., Section II.4 of \cite{ST}). Then $\hat\nu$
is a positive Borel measure of the same mass as $\nu$, which is
supported on $S_w$. Furthermore, we can express $\hat\nu$ via
harmonic measures
\[
\hat\nu=\sum_{j=1}^K p_j \omega(z_j,\cdot,\Omega)
\] (cf. Appendix A.3 of \cite{ST}). The potentials of $\nu$ and $\hat\nu$ are
related by the equation
\begin{equation} \label{5.15}
U^{\hat\nu}(x)=U^{\nu}(x)+C_1, \qquad \mbox{for q.e. } x\in S_w,
\end{equation}
where $C_1$ is a constant. This equation  holds quasi everywhere
on $S_w$, i.e., with the exception of a set of zero logarithmic
capacity (see Theorem II.4.4 of \cite{ST}). Using \eqref{2.5},
(\ref{5.14}) and (\ref{5.15}), we obtain for quasi every $x\in
S_w$ that
\begin{eqnarray*}
F_w &=& U^{\mu_w}(x)-\log w(x) = U^{\mu_w}(x)+U^{\nu}(x)-\log A\\
&=& U^{\mu_w}(x)+U^{\hat\nu}(x)-C_1-\log A
\end{eqnarray*}
Recall that the potential $U^{\omega(\infty,\cdot,\Omega)}(x)$ is
constant q.e. on $S_w$, by Frostman's theorem. Thus we have
\begin{equation} \label{5.16}
U^{\mu_w+\hat\nu}(x) = U^{(1+p)\omega(\infty,\cdot,\Omega)}(x) +
C_2, \qquad \mbox{for q.e. }x\in S_w,
\end{equation}
where $C_2$ is another constant. Note that all measures in the
above equation have finite logarithmic energy. Since the mass of
$\mu_w+\hat\nu$ is equal to that of
$(1+p)\omega(\infty,\cdot,\Omega)$, we can apply the Principle of
Domination (cf. Theorem II.3.2 of \cite{ST}) to conclude that
\eqref{5.16} holds for every $x\in\C.$ The Unicity Theorem (see
Theorem I.3.3 of \cite{ST}) now shows that
\[
\mu_w + \hat\nu=(1+p)\omega(\infty,\cdot,\Omega).
\]

\end{proof}

\begin{proof}[Proof of Theorem \ref{thm2.1}]

It is clear from \eqref{2.5} that $S_w \subset [a,b]\setminus Z.$
Since $-\log w(x)$ is convex on each interval $(z_j,z_{j+1}),\
j=1,\ldots,K-1,$ the intersection of $S_w$ with $(z_j,z_{j+1})$ is
either a closed interval or an empty set (cf. Theorem IV.1.10(b)
of \cite{ST}). Hence $S_w=\bigcup_{l=1}^L [a_l,b_l], \ L\le K-1,$
where we have at most one interval $[a_l,b_l]$ between any
neighboring zeros $z_j$ and $z_{j+1}$ of $w$.

Theorem 1.38 of \cite{DKM} states that for the upper limiting
values of
\begin{align} \label{5.17}
F(z):= \frac{\sqrt{R(z)}}{\pi i} \int_{S_w} \frac{-i (\log
w(t))'\,dt}{\pi (t-z)\sqrt{R(t)}}, \qquad z\in\C\setminus S_w,
\end{align}
we have
\begin{align} \label{5.18}
d\mu_w(x)=\Re F(x)\,dx, \qquad x\in S_w.
\end{align}
We remark that the statement of this theorem in \cite{DKM}
requires $\log w(x)$ be real analytic in a neighborhood of
$[a,b].$ But this is only used to conclude that the support of
$\mu_w$ consists of finitely many intervals, which we already have
anyway. For their analysis leading to \eqref{5.17}-\eqref{5.18},
it is sufficient that $\log w \in C^2([a,b]).$ In fact,
\eqref{5.17}-\eqref{5.18} are obtained by solving the singular
integral equation with Cauchy kernel (by the methods similar to
those of \cite[Chap. 11]{Mus}):
\[
\int\frac{d\mu(t)}{t-x} = (\log w(x))', \qquad x \in S_w,
\]
which arises via differentiation of \eqref{2.5}. In order to
achieve that $\log w \in C^2([a,b]),$ we can modify $w$ in the
small neighborhoods of zeros, outside the compact set
\begin{align} \label{5.19}
S_w^* := \{x\in[a,b]: U^{\mu_{w}}(x)-\log w(x) \le F_{w}\}\
\subset [a,b]\setminus Z,
\end{align}
so that the new weight $\tilde w$ satisfies
\begin{align} \label{5.20}
U^{\mu_{w}}(x)-\log \tilde w(x) > F_{w}, \qquad x\in
[a,b]\setminus S_w^*.
\end{align}
Theorem I.3.3 of \cite{ST} and \eqref{5.19}-\eqref{5.20} then give
that the equilibrium measure $\mu_{\tilde w}=\mu_w$. Hence the
result \eqref{5.17}-\eqref{5.18} of Theorem 1.38 \cite{DKM} is
valid in our case.

With the help of Lemma \ref{lem3.1}, we obtain for
$z\in\C\setminus S_w$ that
\begin{align*}
F(z) &= \frac{\sqrt{R(z)}}{\pi i} \sum_{j=1}^K \frac{p_j}{\pi i}
\int_{S_w} \frac{dt}{(t-z_j)(t-z)\sqrt{R(t)}} \\ &=
\frac{\sqrt{R(z)}}{\pi i} \sum_{j=1}^K \frac{p_j}{\pi i (z-z_j)}
\int_{S_w} \left(\frac{1}{t-z}-\frac{1}{t-z_j}\right)
\frac{dt}{\sqrt{R(t)}} \\ &= \frac{\sqrt{R(z)}}{\pi i}
\sum_{j=1}^K \frac{p_j}{\pi i (z-z_j)} \left( \int_{S_w}
\frac{dt}{(t-z)\sqrt{R(t)}} - \int_{S_w}
\frac{dt}{(t-z_j)\sqrt{R(t)}} \right) \\ &= \frac{\sqrt{R(z)}}{\pi
i} \sum_{j=1}^K \left( \frac{p_j}{(z-z_j)\sqrt{R(z)}} -
\frac{p_j}{(z-z_j)\sqrt{R(z_j)}} \right) \\ &= \frac{1}{\pi i}
\sum_{j=1}^K \frac{p_j}{z-z_j} - \frac{\sqrt{R(z)}}{\pi i}
\sum_{j=1}^K \frac{p_j}{(z-z_j)\sqrt{R(z_j)}}.
\end{align*}
Recall that $\sqrt{R(x)}$ is pure imaginary for $x\in S_w$, and
$\sqrt{R(z_j)}$ is real, $j=1,\ldots,K,$ by \eqref{5.7}. Thus we
have that
\[
\Re F(x) = - \frac{\sqrt{R(x)}}{\pi i} \sum_{j=1}^K
\frac{p_j}{(x-z_j)\sqrt{R(z_j)}}, \qquad x\in S_w.
\]
If we set
\begin{align} \label{5.21}
P(x) := \frac{1}{1+p} \sum_{j=1}^K \frac{p_j}{\sqrt{R(z_j)}}
\prod_{l\neq j} (x-z_l),
\end{align}
then
\begin{align} \label{5.22}
d\mu_{w}(x)=-\frac{(1+p) \sqrt{R(x)}\,P(x)}{\pi i \prod_{j=1}^K
(x-z_j)}\, dx, \qquad x \in S_w,
\end{align}
by \eqref{5.18}. Equation \eqref{2.8} now follows from
\eqref{5.7}. Also, \eqref{5.21} and \eqref{5.7} give
\begin{equation} \label{5.23}
P(z_j) = \frac{p_j \prod_{m\neq j} (z_j-z_m)}{(1+p)\sqrt{R(z_j)}}
= (-1)^{L+l}\, \frac{p_j\prod_{m\neq j}
(z_j-z_m)}{(1+p)\sqrt{|R(z_j)|}},
\end{equation}
where $z_j\in [b_l,a_{l+1}],\ j=1,\ldots,K,$ which proves
\eqref{2.9}. On the other hand, Lemmas \ref{lem3.2} and
\ref{lem3.3} yield another representation for $\mu_w$, of the form
\[
d\mu_w(x) = \frac{T(x)\,dx}{\pi i \sqrt{R(x)}\prod_{j=1}^K
(x-z_j)}, \qquad x \in S_w,
\]
where $T\in\R_{K+L-1}[x].$ Comparing this with \eqref{5.22}, we
conclude that
\[
-(1+p) R(x) P(x) = T(x)
\]
and $2L+\deg(P) = \deg(T) \le K+L-1.$ Thus the actual degree of
$P$ satisfies
\[
\deg(P) \le K-L-1.
\]
We now prove that $P$ has leading coefficient $1,$ by using the
following identity:
\begin{align} \label{5.24}
\lim_{x\to\infty} (-x) \int_{S_w} \frac{d\mu_{w}(t)}{t-x} =
\mu_w(\C) = 1.
\end{align}
Observe that $\sqrt{R(t)}\,P(t)/\left((t-x)\prod_{j=1}^K
(t-z_j)\right)$ is an analytic function of $t$ in
$\overline\C\setminus S_w$, except for the simple poles at $x$ and
$z_j,\ j=1,\ldots,K.$ Applying Cauchy's integral theorem, we have
for a small $r>0$:
\begin{align*}
\int_{S_w} \frac{d\mu_{w}(t)}{t-x} &= -\frac{1+p}{\pi i}
\int_{S_w} \frac{\sqrt{R(t)}\,P(t)\,dt}{(t-x)\prod_{j=1}^K
(t-z_j)} = -\frac{1+p}{2 \pi i} \oint_{S_w}
\frac{\sqrt{R(t)}\,P(t)\,dt}{(t-x)\prod_{j=1}^K (t-z_j)} \\
&= -\frac{1+p}{2 \pi i} \left(\oint_{|t-x|=r} + \sum _{l=1}^K
\oint_{|t-z_l|=r} \right)
\frac{\sqrt{R(t)}\,P(t)\,dt}{(t-x)\prod_{j=1}^K (t-z_j)} \\
&= -(1+p) \frac{\sqrt{R(x)}\,P(x)}{\prod_{j=1}^K (x-z_j)} - (1+p)
\sum _{l=1}^K \frac{\sqrt{R(z_l)}\,P(z_l)}{(z_l-x)\prod_{j\neq l}
(z_l-z_j)}.
\end{align*}
Taking into account \eqref{5.23}, we obtain that
\begin{align} \label{5.25}
\int_{S_w} \frac{d\mu_{w}(t)}{t-x} = -(1+p)
\frac{\sqrt{R(x)}\,P(x)}{\prod_{j=1}^K (x-z_j)} - \sum _{l=1}^K
\frac{p_l}{z_l-x}, \quad x\in\C\setminus(S_w\cup Z).
\end{align}
If $c_{K-L-1}$ is the leading coefficient of $P$, then
\[
\lim_{x\to\infty} (-x) \int_{S_w} \frac{d\mu_{w}(t)}{t-x} =
(1+p)c_{K-L-1} - \sum_{l=1}^K p_l = 1,
\]
by \eqref{5.24}. Thus $(1+p)c_{K-L-1}=1+p$ and
\[
c_{K-L-1} = 1.
\]

It only remains to prove \eqref{2.10} now. We find from
\eqref{2.5} that
\begin{align} \label{5.26}
U^{\mu_w}(a_{l+1}) - U^{\mu_w}(b_l) = \log w(a_{l+1}) - \log
w(b_l), \quad l=1,\ldots,L-1.
\end{align}
The potential $U^{\mu_w}(x)$ is continuous in $\C$ by Theorem
I.4.8 of \cite{ST}, and it is infinitely differentiable on
$(b_l,a_{l+1})$. Hence we obtain by the fundamental theorem of
calculus and \eqref{5.25} that
\begin{align*}
U^{\mu_w}(a_{l+1}) - U^{\mu_w}(b_l) &= \int_{b_l}^{a_{l+1}}
\frac{d}{dx}\left(U^{\mu_w}(x)\right)\, dx = \int_{b_l}^{a_{l+1}}
\int_{S_w} \frac{d\mu_{w}(t)}{t-x}\, dx \\ &= \int_{b_l}^{a_{l+1}}
\left(-(1+p) \frac{\sqrt{R(x)}\,P(x)}{\prod_{j=1}^K (x-z_j)} +
\sum_{j=1}^K \frac{p_j}{x-z_j}\right)\, dx \\
&= -(1+p) \int_{b_l}^{a_{l+1}}
\frac{\sqrt{R(x)}\,P(x)\,dx}{\prod_{j=1}^K (x-z_j)} \\ &+
\sum_{j=1}^K p_j\left(\log|a_{l+1}-z_j| - \log|b_l-z_j| \right),
\end{align*}
where the last equality holds in the principal value sense.
Therefore, we have by \eqref{5.26} that
\[
-(1+p) \int_{b_l}^{a_{l+1}}
\frac{\sqrt{R(x)}\,P(x)\,dx}{\prod_{j=1}^K (x-z_j)} +
\log\frac{w(a_{l+1})}{w(b_l)} = \log\frac{w(a_{l+1})}{w(b_l)},
\]
which proves \eqref{2.10}.

\end{proof}

\begin{proof}[Proof of Corollary \ref{cor2.2}]

If $L=K-1$  then obviously $P(x)\equiv 1$ and $z_j\in
(b_{j-1},a_j),\ j=1,\ldots,K.$ Hence
\[
\prod_{j=1}^K (x-z_j)=(-1)^{L+1-l} \prod_{j=1}^K |x-z_j|, \qquad
x\in [a_l,b_l],
\]
and \eqref{2.11} follows from \eqref{2.8}. Also, we have
\[
\prod_{m\neq j} (z_j-z_m) = (-1)^{L+1-j} \prod_{m\neq j}
|z_j-z_m|,
\]
which gives \eqref{2.12} by \eqref{2.9}. Equation \eqref{2.13} is
an immediate consequence of \eqref{2.10}.

\end{proof}

\begin{proof}[Proof of Remark \ref{rem2.3}]

Consider the weight $w(x)=(1+x)^{p_1} |x|^{p_2} (1-x)^{p_3}$ on
$[-1,1]$, where we take $p_2=p_3=1$. We show that the support
$S_w$, for large $p_1$, consists of only one interval. Assume to
the contrary that $S_w=[a_1,b_1]\cup [a_2,b_2].$ Then Corollary
\ref{cor2.2} applies here, and \eqref{2.12}-\eqref{2.13} hold
true. Note that these equations are identical to the equations of
Amoroso \cite[p. 1184]{Amo}, used to define the endpoints
$a_1,b_1,a_2,b_2$. As it turns out, they do not have a solution
for large $p_1$. Indeed, \eqref{2.13} gives
\begin{align} \label{5.27}
\int_{b_1}^{a_2}
\frac{\sqrt{(x-a_1)(x-b_1)(a_2-x)(b_2-x)}}{x(x^2-1)}\,dx = 0,
\end{align}
and \eqref{2.12} gives for $z_1=-1$ that
\[
\sqrt{(1+a_1)(1+b_1)(1+a_2)(1+b_2)} = \frac{2 p_1}{p_1+3}.
\]
Since $-1<a_1<b_1<0<a_2<b_2<1,$ we have from the latter equation
that
\[
\sqrt{(1+a_2)(1+b_2)} > \frac{2 p_1}{p_1+3},
\]
and
\[
\lim_{p_1\to+\infty} a_2 = \lim_{p_1\to+\infty} b_2 = 1.
\]
Similarly,
\[
2\sqrt{(1+a_1)(1+b_1)} > \frac{2 p_1}{p_1+3}
\]
implies that
\[
\lim_{p_1\to+\infty} a_1 = \lim_{p_1\to+\infty} b_1 = 0.
\]
Therefore $\sqrt{(x-a_1)(x-b_1)(a_2-x)(b_2-x)}$ is strictly
increasing on $[b_1,-b_1]$, for all sufficiently large $p_1$, and
\[
\int_{b_1}^{-b_1}
\frac{\sqrt{(x-a_1)(x-b_1)(a_2-x)(b_2-x)}}{x(x^2-1)}\,dx < 0.
\]
Coupling this with the obvious inequality
\[
\int_{-b_1}^{a_2}
\frac{\sqrt{(x-a_1)(x-b_1)(a_2-x)(b_2-x)}}{x(x^2-1)}\,dx < 0,
\]
we obtain a contradiction to \eqref{5.27}.

It is possible to show that the number of intervals of the support
for $\mu_w$ in Theorem \ref{thm2.1} can indeed take any value
between $1$ and $K-1.$

\end{proof}

\begin{proof}[Proof of Corollary \ref{cor2.4}]

This result follows from Corollary \ref{2.2}. The representation
\eqref{2.14} for $\mu_w$  is an immediate consequence of
\eqref{2.11}. We also obtain from \eqref{2.12} that the endpoints
of the support must satisfy the equations
\[
a_1 b_1=r_1^2 \quad \mbox{and} \quad (1-a_1)(1-b_1)=r_2^2.
\]
Solving those equations, we obtain \eqref{2.15} and \eqref{2.16}.

\end{proof}

\begin{proof}[Proof of Lemma \ref{lem2.1}]

Consider the weighted Fekete points $\{\zeta_i^{(n)}\}_{i=1}^n
\subset [a,b], \ n\in\N,$ that maximize the absolute value of the
weighted Vandermonde determinant \eqref{1.10}. The relation
between the problems of minimizing energy \eqref{2.2}-\eqref{2.3}
and maximizing \eqref{1.10} becomes transparent if we consider
$-\frac{2}{n(n-1)}\log|V_n^w|$, which is essentially a discrete
version of the weighted energy functional \eqref{2.2}. Indeed, the
normalized counting measures
\[ \nu_n := \frac{1}{n} \sum_{i=1}^n \delta_{\zeta_i^{(n)}} \]
converge weakly to $\mu_w$, the extremal measure for
\eqref{2.2}-\eqref{2.3}, and \eqref{2.30} holds true (cf. Section
III.1 of \cite{ST}). Thus the discrete problem is a good
approximation of the continuous one.

We deduce \eqref{2.31} from the results of G\"otz and Saff
\cite{GSa}. They require that $\log w(x)$ be H\"older continuous
on $[a,b]$, which is not true if $w$ of \eqref{2.1} has zeros on
$[a,b]$. But we can modify $w$ in the small neighborhoods of those
zeros, outside the compact set
\begin{align} \label{4.1}
S_w^* = \{x\in[a,b]: U^{\mu_{w}}(x)-\log w(x) \le F_{w}\}\ \subset
[a,b]\setminus Z,
\end{align}
so that the new weight $\tilde w$ satisfies
\begin{align} \label{4.2}
U^{\mu_{w}}(x)-\log \tilde w(x) > F_{w}, \qquad x\in
[a,b]\setminus S_w^*,
\end{align}
and $\log\tilde w(x)$ is H\"older continuous on $[a,b].$ It
follows from Theorem I.3.3 of \cite{ST} and
\eqref{4.1}-\eqref{4.2} that the equilibrium measure $\mu_{\tilde
w} =\mu_w$, and from Theorem III.1.2 of \cite{ST} and \eqref{4.1}
that the weighted Fekete points for $\tilde w$ are identical to
those for $w$. Thus all results of \cite{GSa} are applicable here.
Set
\[
F_n(x) := \prod_{i=1}^n (x-\zeta_i^{(n)}), \qquad n\in\N.
\]
Then $\frac{1}{n}\log|F_n(x)|=-U^{\nu_n}(x)$ and
\[
U^{\mu_w}(x)+\frac{1}{n}\log|F_n(x)| \le c_0\,\frac{\log n}{n},
\qquad x\in\C,
\]
where $c_0>0$ depends only on $w$, by Theorem 1 of \cite{GSa}.
Hence
\begin{align} \label{4.3}
|F_n(x)| \le n^{c_0}\, e^{-n\, U^{\mu_{w}}(x)}, \qquad x\in\C.
\end{align}
For a small $r>0$, write
\[
F_n'(x) = \frac{1}{2\pi i} \int_{|x-z|=r}
\frac{F_n(z)\,dz}{(z-x)^2}, \qquad x \in S_w^*,
\]
and estimate
\[
w^{n-1}(x)|F_n'(x)| \le \frac{w^{n-1}(x)}{r} \max_{|x-z|=r}
|F_n(z)| = O(n^{c_0})\, \frac{w^n(x)}{r} \max_{|x-z|=r} e^{-n\,
U^{\mu_{w}}(z)},
\]
as $n\to \infty$, by \eqref{4.3}. Note that $U^{\mu_{w}}(x)$ is
H\"older continuous in $\C$, because it is a harmonic function in
$\C\setminus S_w$, with smooth boundary values $\log w(x) + F_w$
on $S_w$ (see Theorem I.4.7 of \cite{ST} and Lemma 2 of
\cite{GSa}). If $\lambda\in(0,1]$ is the H\"older exponent for
$U^{\mu_{w}}(x)$, then we choose $r=n^{-1/\lambda}$ and obtain
\[
\max_{|x-z|=n^{-1/\lambda}} e^{-n\, U^{\mu_{w}}(z)} = O(1)\,
e^{-n\, U^{\mu_{w}}(x)}, \qquad x\in S_w^*.
\]
Hence
\begin{align} \label{4.4}
w^{n-1}(x)|F_n'(x)| &= O(n^{c_0+1/\lambda})\, e^{n(\log w(x) -
U^{\mu_{w}}(x))} \\ \nonumber &= O(n^{c_0+1/\lambda})\,
e^{-n\,F_w}, \qquad x\in S_w^*,
\end{align}
by \eqref{4.1} \and \eqref{2.4}. Recall that the weighted Fekete
points are contained in the compact set $S_w^* \subset
[a,b]\setminus Z$ (cf. Theorem III.1.2 of \cite{ST}), where $\log
w(x)$ is continuous. Therefore, we have from Theorem 3 of
\cite{GSa} that
\[
\int\log w\, d\nu_n - \int\log w\, d\mu_w = O\left(\frac{\log^2
n}{n}\right) \qquad \mbox{as } n\to\infty.
\]
This implies
\begin{align} \label{4.5}
\prod_{i=1}^n w(\zeta_i^{(n)}) = O(e^{\log^2 n})\, e^{n\int\log
w\, d\mu_w} \qquad \mbox{as } n\to\infty.
\end{align}
Observe that
\[
(V^w_n(\zeta_1^{(n)},\ldots,\zeta_n^{(n)}))^2 = \prod_{i=1}^n
w^{2(n-1)}(\zeta_i^{(n)}) \prod_{i=1}^n F_n'(\zeta_i^{(n)}).
\]
We now use \eqref{4.4} and \eqref{4.5} to estimate
\begin{align*}
\max_{[a,b]^n} (V_n^w)^2 &= \prod_{i=1}^n w^{n-1}(\zeta_i^{(n)})
\prod_{i=1}^n w^{n-1}(\zeta_i^{(n)}) |F_n'(\zeta_i^{(n)})| \\ &=
O(d^{n\,\log^2 n})\, e^{n(n-1)(\int\log w\, d\mu_w - F_w)} \\
&= O(d^{n\,\log^2 n})\, e^{-n(n-1)V_w} = O(d^{n\,\log^2 n})\,
\left(\textup{cap}([a,b],w)\right)^{n(n-1)},
\end{align*}
where $d>e$, as $n\to\infty.$ Thus the upper bound in \eqref{2.31}
is proved. The lower bound of \eqref{2.31} is a well known
consequence of extremal properties for the weighted Fekete points
and Vandermonde determinants, see Theorem III.1.1 of \cite{ST},
which states that the sequence
$|V^w_n(\zeta_1^{(n)},\ldots,\zeta_n^{(n)})|^{\frac{2}{n(n-1)}}$
decreases to $\textup{cap}([a,b],w)$ as $n\to\infty.$

\end{proof}

{\bf Acknowledgement.} The author thanks Amit Ghosh for several
helpful conversations about this paper, and the referee for
suggested improvements.


\begin{thebibliography}{99}

\bibitem{Amo} F. Amoroso, {\it $f$-transfinite diameter and
number theoretic applications}, Ann. Inst. Fourier (Grenoble) {\bf
43} (1993), 1179-1198.
\bibitem{AB} V. V. Andrievskii and H.-P. Blatt, Discrepancy of
Signed Measures and Polynomial Approximation, Springer-Verlag, New
York, 2002.
\bibitem{BC99} T. Bloom and J. P. Calvi, {\it On the multivariate
transfinite diameter}, Ann. Polon. Math. {\bf 72} (1999), 285-305.
\bibitem{BC00} T. Bloom and J. P. Calvi, {\it On multivariate
minimal polynomials}, Math. Proc. Cambridge Philos. Soc. {\bf 129}
(2000), 417-431.
\bibitem{Bor} P. Borwein, Computational Excursions in Analysis and
Number Theory, Springer-Verlag, New York, 2002.
\bibitem{Che} P. L. Chebyshev, Collected Works, Vol. 1, Akad.
Nauk SSSR, Moscow, 1944. (Russian)
\bibitem{Chu} G. V. Chudnovsky, {\it Number theoretic applications of
polynomials with rational coefficients defined by extremality
conditions}, Arithmetic and Geometry, Vol. I, M. Artin and J. Tate
(eds.), Birkh\"{a}user, Boston, 1983, pp. 61-105.
\bibitem{Dav} H. Davenport, Multiplicative Number Theory,
Springer-Verlag, New York, 1980.
\bibitem{DKM} P. Deift, T. Kreicherbauer and K. T.-R. McLaughlin,
{\it New results on the equilibrium measure for logarithmic
potentials in the presence of an external field}, J. Approx.
Theory {\bf 95} (1998), 388-475.
\bibitem{Dia} H. G. Diamond, {\it Elementary methods in the
study of the distribution of prime numbers}, Bull. Amer. Math.
Soc. {\bf 7} (1982), 553-589.
\bibitem{Erd} P. Erd\H{o}s, {\it On a new method in elementary
number theory which leads to an elementary proof of the prime
number theorem}, Proc. Nat. Acad. Sci. U.S.A. {\bf 35} (1949),
374-384.
\bibitem{Fek} M. Fekete, {\it \"{U}ber die Verteilung der Wurzeln bei
gewissen algebraischen Gleichungen mit ganzzahligen
Koeffizienten}, Math. Zeit. {\bf 17} (1923), 228-249.
\bibitem{Fro} O. Frostman, {\it La m\'ethode de variation de Gauss
et les fonctions sousharmoniques}, Acta Sci. Math. {\bf 8}
(1936-37), 149-159.
\bibitem{Gau} C. F. Gauss, {\it Allgemeine Lehrs\"atze in
Beziehung auf die in verkehrten Verh\"altnissen des Quadrats der
Entfernung wirkenden Anzienhungs- und Abstossungs- Kr\"afte}, in
Werke, Band 5, G\"ottingen, 1840, pp. 197-242.
\bibitem{Gor} D. S. Gorshkov, {\it On the distance from zero
on the interval $[0,1]$ of polynomials with integral
coefficients}, in ``Proc. of the Third All Union Mathematical
Congress" (Moscow, 1956), Vol. 4, Akad. Nauk SSSR, Moscow, 1959,
pp. 5-7. (Russian)
\bibitem{GSa} M. G\"otz and E. B. Saff, {\it Potential and
discrepancy estimates for weighted extremal points}, Constr.
Approx. {\bf 16} (2000), 541-557.
\bibitem{Ing} A. E. Ingham, The Distribution of Prime Numbers,
Cambridge Univ. Press, London, 1932.
\bibitem{Kli} M. Klimek, Pluripotential Theory, Oxford Univ.
Press, Oxford, 1991.
\bibitem{Mon} H. L. Montgomery, Ten Lectures on the Interface
Between Analytic Number Theory and Harmonic Analysis, CBMS, Vol.
84, Amer. Math. Soc., Providence, R.I., 1994.
\bibitem{Mus} N. I. Muskhelishvili, Singular Integral Equations,
Dover, New York, 1992.
\bibitem{Nai} M. Nair, {\it A new method in elementary prime
number theory}, J. London Math. Soc. {\bf 25} (1982), 385-391.
\bibitem{Pri} I. E. Pritsker, {\it Small polynomials with integer
coefficients}, Proc. London Math. Soc. (submitted); Available
electronically at http://www.math.okstate.edu/\~{}igor/intcheb.ps
\bibitem{Ran} T. Ransford, Potential Theory in the Complex Plane,
Cambridge Univ. Press, Cambridge, 1995.
\bibitem{Rie} B. Riemann, {\it Ueber die Anzahl der Primzahlen
unter einer gegebenen Gr\"osse}, in Werke, Teubner, Leipzig, 1892,
pp. 145-155.
\bibitem{ST} E. B. Saff and V. Totik, Logarithmic Potentials with
External Fields, Springer-Verlag, Berlin, 1997.
\bibitem{Sel} A. Selberg, {\it An elementary proof of prime number
theorem}, Ann. of Math. {\bf 50} (1949), 305-313.
\bibitem{StT} H. Stahl and V. Totik, General Orthogonal Polynomials,
Cambridge Univ. Press, Cambridge, 1992.
\bibitem{Sze} G. Szeg\H{o}, {\it Bemerkungen zu einer Arbeit von Herrn M.
Fekete: \"{U}ber die Verteilung der Wurzeln bei gewissen
algebraischen Gleichungen mit ganzzahligen Koeffizienten}, Math.
Zeit. {\bf 21} (1924), 203-208.
\bibitem{TF} G. Tenenbaum and M. M. France, The Prime Numbers and
Their Distribution, AMS, Providence, 2000.
\bibitem{Tot} V. Totik, {\it Polynomial inverse images and
polynomial inequalities}, Acta Math. {\bf 187} (2001), 139-160.
\bibitem{Tri} R. M. Trigub, {\it Approximation of functions
with Diophantine conditions by polynomials with integral
coefficients}, in ``Metric Questions of the Theory of Functions
and Mappings," No. 2, Naukova Dumka, Kiev, 1971, pp. 267-333.
(Russian)
\bibitem{Zah} V. P. Zaharjuta, {\it Transfinite diameter,
Chebyshev constants, and capacity for compacta in $\C^n$}, Math.
USSR-Sb. {\bf 25} (1975), 350-364.

\end{thebibliography}
\end{document}